\documentclass[12pt]{article}

\usepackage{amsfonts}
\usepackage{amssymb}
\usepackage{amsthm}
\usepackage{amsmath}
\usepackage{graphicx}
\usepackage[lofdepth,lotdepth]{subfig}
\usepackage{times}
\usepackage{rotating}
\usepackage{enumerate}

\newtheorem{thm}{Theorem}[section]
\newtheorem{prop}[thm]{Proposition}

\newtheorem{lem}[thm]{Lemma}

\theoremstyle{definition}
\newtheorem{example}[thm]{Example}
\newtheorem{definition}[thm]{Definition}

\newtheorem*{ack}{Acknowledgements}

\newcommand{\Z}{\mathbb Z}
\newcommand{\Comment}[1]{}

\DeclareMathOperator{\facets}{Facets}
\DeclareMathOperator{\N}{\mathcal{N}}
\DeclareMathOperator{\F}{\mathcal{F}}

\begin{document}

\title{A criterion for a monomial ideal to have a linear resolution in characteristic $2$}

\author{ E. Connon\thanks{Research supported by a Killam scholarship.} \qquad \qquad S. Faridi\thanks{Research supported by NSERC.} }

\maketitle

\begin{center} \it \small
Department of Mathematics and Statistics, Dalhousie University, Canada
\end{center}

\begin{abstract}
In this paper we give a necessary and sufficient combinatorial condition for a monomial ideal to have a linear resolution over fields of characteristic $2$.  We also give a new proof of Fr\"oberg's theorem over fields of characteristic $2$.
\end{abstract}

{\bf Keywords:} \ linear resolution, monomial ideal, chordal graph, simplicial complex, simplicial homology, Stanley-Reisner complex, edge ideal, facet ideal, chorded complex, chordal hypergraph, Fr\"oberg's theorem\\


\section{Introduction}

Recently there has been interest in finding a characterization of square-free monomial ideals with linear resolutions in terms of the combinatorics of their associated simplicial complexes or hypergraphs.  See, for example, \cite{Em10}, \cite{HaVT08}, \cite{MNYZ12}, \cite{MYZ12}, and \cite{Wood11}.  This exploration was motivated by a theorem of Fr\"oberg from \cite{Fr90} in which he gives the following combinatorial classification of the square-free monomial ideals generated in degree two which have linear resolutions.

\begin{thm}[Fr\"oberg \cite{Fr90}] \label{thm:Frob_thm}
The edge ideal of a graph $G$ has a linear resolution if and only if the complement of $G$ is chordal.
\end{thm}

This characterization has inspired the introduction of several different definitions of a ``chordal'' hypergraph with the goal of achieving a generalization of Fr\"oberg's theorem to higher-dimensions.  Emtander \cite{Em10} and Woodroofe \cite{Wood11} use their respective definitions of a ``chordal'' hypergraph to give a sufficient condition for a square-free monomial ideal to have a linear resolution over all fields.  In \cite{ConFar12}, the authors introduce the notion of a {\bf $d$-chorded} simplicial complex and use it to give a necessary combinatorial condition for an ideal to have a linear resolution over all fields.

Obtaining a complete generalization of Fr\"oberg's theorem to higher dimensions is made difficult by the fact that there exist square-free monomial ideals which have linear resolutions over some fields and not others.  In particular the existence of a linear resolution depends on the characteristic of the field.  The Stanley-Reisner ideal of the triangulation of the real projective plane is a typical example and has a linear resolution only over fields of characteristic not equal to $2$.  Such examples tell us that when an ideal is generated in degrees higher than two it is not always the combinatorics of the associated simplicial complex that determines the existence of a linear resolution.  In this paper we concentrate on fields of characteristic $2$ because in this case we have a more direct relationship between the combinatorics of a complex and its simplicial homology (see \cite{Con13}) which is of primary interest when determining the existence of a linear resolution.

The condition we give in \cite{ConFar12} is not sufficient to ensure linear resolution. In this paper we are able to characterize the obstructions to the converse over fields of characteristic $2$ by demonstrating that all counter-examples share a specific combinatorial property.  In Section \ref{sec:criterion_char2} we are able to give the following necessary and sufficient condition for an ideal to have a linear resolution over fields of characteristic $2$ based on the combinatorial structure of the Stanley-Reisner complex of the ideal.

\begin{thm}
Let $I$ be generated by square-free monomials in the same degree.  Then $I$ has a linear resolution over fields of characteristic $2$ if and only if the Stanley-Reisner complex of $I$ is chorded.
\end{thm}

In Section \ref{sec:Frob_proof} we give a new combinatorial proof of Theorem \ref{thm:Frob_thm} over fields of characteristic $2$.

\begin{ack}
The authors would like to thank Rashid Zaare-Nahandi whose comments inspired the writing of this paper, and MSRI for their hospitality during the preparation of this paper.
\end{ack}


\section{Background}
Let $k$ be a field and let $R = k[x_1,\ldots,x_n]$.  For any monomial ideal $I$ in $R$ there is a {\bf minimal graded free resolution} of $I$ of the form
{
\[
	0 \rightarrow \bigoplus_j R(-j)^{\beta_{m,j}(I)} \rightarrow \bigoplus_j R(-j)^{\beta_{m-1,j}(I)} \rightarrow \cdots \rightarrow \bigoplus_j R(-j)^{\beta_{0,j}(I)} \rightarrow I \rightarrow 0
\]
}
where $R(-j)$ denotes the free $R$-module obtained by shifting the degrees of $R$ by $j$ and $m \leq n$.  The numbers $\beta_{i,j}(I)$ are called the {\bf graded Betti numbers} of $I$.  We say that $I$ has a {\bf $d$-linear resolution over $k$} if $\beta_{i,j}(I)=0$ for all $j \neq i+d$.  It follows that $I$ is generated in degree $d$.

It is known that classifying monomial ideals with linear resolutions is equivalent to classifying Cohen-Macaulay monomial ideals and that it is sufficient to consider square-free monomials \cite{EagRein98,Fr85}.

 By studying square-free monomial ideals we are able to make use of techniques from Stanley-Reisner theory and facet ideal theory by associating our ideal to a combinatorial object.  Recall that an (abstract) {\bf simplicial complex} $\Gamma$ on the finite set of {\bf vertices} $V(\Gamma)$ is a collection of subsets of $V(\Gamma)$ called {\bf faces} or {\bf simplices} such that if $F \in \Gamma$ and $F' \subseteq F$ then $F' \in \Gamma$.  The faces of $\Gamma$ that are not strictly contained in any other face of $\Gamma$ are called {\bf facets} and we denote the facet set by $\facets(\Gamma)$.  If $\facets(\Gamma)=\{F_1,\ldots,F_k\}$ then we write
\[
    \Gamma = \langle F_1,\ldots,F_k \rangle.
\]

The {\bf dimension} of a face $F$ of $\Gamma$ is equal to $|F|-1$.  A face of $\Gamma$ of dimension $d$ is referred to as a {\bf $d$-face}.  The {\bf dimension} of the simplicial complex $\Gamma$, denoted by $\dim \Gamma$, is the maximum dimension of its facets.  The complex $\Gamma$ is {\bf pure} if these facets all share the same dimension.

The {\bf pure $d$-skeleton} of a simplicial complex $\Gamma$, written $\Gamma^{[d]}$, is the simplicial complex whose facets are the faces of $\Gamma$ of dimension $d$. The complex $\Gamma$ is said to be {\bf $d$-complete} if all possible $d$-faces are present in $\Gamma$. The {\bf $d$-complement} of $\Gamma$ is the complex $\overline{\Gamma}_d$ with
\[
    \facets(\overline{\Gamma}_d) = \{F \subset V(\Gamma) \ | \ |F|=d+1, F \notin \Gamma\}.
\]

The {\bf induced subcomplex} of $\Gamma$ on the vertex set $S\subseteq V(\Gamma)$, denoted $\Gamma_S$, is the simplicial complex whose faces are those faces of $\Gamma$ contained in $S$.

A pure $d$-dimensional simplicial complex is {\bf $d$-path-connected} when each pair of $d$-dimensional faces are joined by a sequence of $d$-dimensional faces whose intersections are $(d-1)$-faces.  The {\bf $d$-path-connected components} of a pure $d$-dimensional simplicial complex are the maximal subcomplexes which are $d$-path-connected.

The {\bf Stanley-Reisner complex} of the square-free monomial ideal $I$ in the polynomial ring $k[x_1,\ldots,x_n]$ is the simplicial complex on the vertices $x_1,\ldots,x_n$ whose faces are given by the monomials not belonging to $I$.  It is denoted $\N(I)$.  Conversely, the {\bf Stanley-Reisner ideal} of the simplicial complex $\Gamma$, denoted $\N(\Gamma)$, is the ideal generated by monomials $x_{i_1}x_{i_2}\cdots x_{i_k}$ such that $\{x_{i_1},x_{i_2},\ldots,x_{i_k}\}$ is not a face of $\Gamma$.  See Figure \ref{fig:SR_complex} for an example of this relationship.

\begin{figure}[h!]
{\centering
    \includegraphics[height=1.15in]{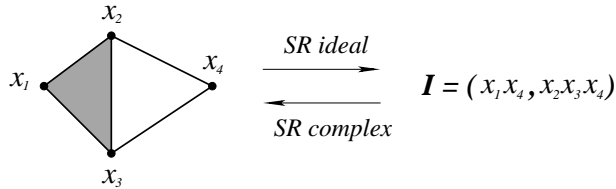}
    \caption {Stanley-Reisner relationship} \label{fig:SR_complex}

}
\end{figure}

The {\bf facet complex} of the square-free monomial ideal $I$ in $k[x_1,\ldots,x_n]$ is the simplicial complex $\F(I)$ on the vertices $x_1,\ldots,x_n$ whose facets are given by the minimal monomial generators of $I$.  The {\bf facet ideal} of the complex $\Gamma$ is generated by the monomials $x_{i_1}x_{i_2}\cdots x_{i_k}$ such that $\{x_{i_1},x_{i_2},\cdots ,x_{i_k}\}$ is a facet of $\Gamma$.  The facet ideal of $\Gamma$ is denoted by $\F(\Gamma)$.  An example is given in Figure \ref{fig:facet_complex}.

\begin{figure}[h!]
{\centering
    \includegraphics[height=1.5in]{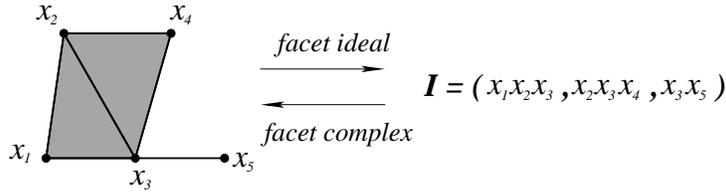}
    \caption {Facet ideal relationship} \label{fig:facet_complex}

}
\end{figure}

In \cite{Fr85} Fr\"oberg shows that a square-free monomial ideal has a linear resolution if and only if the simplicial homology of its Stanley-Reisner complex and of its induced subcomplexes vanish in all but one dimension.

\begin{thm}[Fr\"oberg \cite{Fr85}] \label{thm:Frob_hom_char}
A square-free monomial ideal $I$ has a $t$-linear resolution over a field $k$ if and only if $\tilde{H}_i((\N(I))_S;k) = 0$ for all $S \subseteq V(\N(I))$ and $i \neq t-2$.
\end{thm}

Therefore one way to discover which square-free monomial ideals have linear resolutions is to examine the simplicial homology of their Stanley-Reisner complexes.  In \cite{Con13} it is shown that non-vanishing $d$-dimensional simplicial homology in characteristic $2$ is equivalent to to the presence of a particular combinatorial structure in the simplicial complex called a {\bf $d$-dimensional cycle}.

\begin{definition} [{\bf $d$-dimensional cycle}]
For $d\geq 1$, a {\bf $d$-dimensional cycle} is a pure $d$-dimensional simplicial complex which is $d$-path-connected and has the property that each of its $(d-1)$-dimensional faces is contained in an even number of its $d$-dimensional faces.
\end{definition}

The concept of a $d$-dimensional cycle can be thought of as a generalization of the graph cycle to higher dimensions with similar homological behaviour. Recall that the {\bf support complex} of a homological $d$-chain $c=\alpha_1F_1+\cdots +\alpha_qF_q$, where each $\alpha_i$ is a non-zero element of the field $k$ under consideration, is the simplicial complex $\langle F_1,\ldots,F_q \rangle$.  The $d$-chain $c$ is a {\bf homological $d$-cycle} if it belongs to the kernel of the $d$-boundary operator so that $\partial_d(c)=0$.  It is considered a {\bf $d$-boundary} if $\partial_{d+1}(c')=c$ for some $(d+1)$-chain $c'$.  In this paper we will make use of the following close relationship between $d$-dimensional cyles and homological $d$-cycles over the field $\Z_2$.

\begin{prop} [Connon \cite{Con13}] \label{prop:ddimcycle_is_dcycle}
The sum of the $d$-faces of a $d$-dimensional cycle is a homological $d$-cycle over $\Z_2$ and, conversely, the $d$-path-connected components of the support complex of a homological $d$-cycle are $d$-dimensional cycles.
\end{prop}

The following two propositions provide ways of building higher and lower-dimensional cycles from a $d$-dimensional cycle.

\begin{prop} [Connon \cite{Con13}] \label{prop:cone_cycle}
Let $\Omega$ be a $d$-dimensional cycle with $d$-faces $F_1,\ldots,F_k$ in a simplicial complex $\Gamma$.  Suppose that there exist $(d+1)$-faces $A_1,\ldots,A_\ell$ in $\Gamma_{V(\Omega)}$ such that, over $\Z_2$ we have
\begin{equation}\label{eq:boundaryofcycle}
    \partial_{d+1}\left(\sum_{i=1}^\ell A_i\right)=\sum_{j=1}^k F_j
\end{equation}
and for no strict subset of $\{A_1,\ldots,A_\ell\}$ does (\ref{eq:boundaryofcycle}) hold.  Let $v$ be a vertex with $v \notin V(\Omega)$ and let $\Phi =\langle F_1 \cup v,\ldots,F_k\cup v, A_1,\ldots,A_\ell \rangle$ then $\Phi$ is a $(d+1)$-dimensional cycle.
\end{prop}

\begin{prop} [Connon \cite{Con13}] \label{prop:dcycle_contains_smallercycle}
Let $\Omega$ be a $d$-dimensional cycle and let $v \in V(\Omega)$.  If $F_1,\ldots,F_k$ are the $d$-faces of $\Omega$ which contain $v$ then the $(d-1)$-path-connected components of the complex $\langle F_1\setminus \{v\}, \ldots, F_k\setminus \{v\} \rangle$ are $(d-1)$-dimensional cycles.
\end{prop}

A $d$-dimensional cycle is called {\bf face-minimal} if no strict subset of its $d$-dimensional faces also forms a $d$-dimensional cycle.  An example of a face-minimal $2$-dimensional cycle, the hollow tetrahedron, is given in Figure \ref{fig:hollow_tetra}.

\begin{figure}[h!]
{\centering
    \includegraphics[height=.8in]{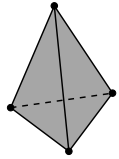}
    \caption {A $2$-complete face-minimal $2$-dimensional cycle} \label{fig:hollow_tetra}

}
\end{figure}

To generalize Fr\"oberg's criterion, we develop a higher-dimensional counterpart to chordal graphs.

\begin{definition} [{\bf chord set, $d$-chorded, chorded} \cite{ConFar12}]
Given a $d$-dimensional cycle $\Omega$ in a simplicial complex $\Gamma$ a {\bf chord set} of $\Omega$ in $\Gamma$ is a set $C$ of $d$-dimensional faces in $\Gamma$ not belonging to $\Omega$ which satisfy the following properties:
\begin{enumerate}
\item the simplicial complex whose set of facets is $C \cup \facets(\Omega)$ consists of $k$ $d$-dimensional cycles $\Omega_1,\ldots,\Omega_k$ for $k\geq 2$
\item each $d$-face in $C$ is contained in an even number of the cycles $\Omega_1,\ldots,\Omega_k$,
\item each $d$-face of $\Omega$ is contained in an odd number of the cycles $\Omega_1,\ldots,\Omega_k$,
\item $|V(\Omega_i)|<|V(\Omega)|$ for $i=1,\ldots,k$.
\end{enumerate}
A simplicial complex $\Gamma$ is {\bf $d$-chorded} if it is pure of dimension $d \geq 1$ and all face-minimal $d$-dimensional cycles in $\Gamma$ which are not $d$-complete have a chord set in $\Gamma$.  We say that an arbitrary simplicial complex $\Gamma$ is {\bf chorded} if $\Gamma^{[d]}$ is $d$-chorded for all $1 \leq d\leq \dim \Gamma$.
\end{definition}

As a consequence of the properties of a chord set all face-minimal $d$-dimensional cycles in a $d$-chorded complex can be broken down into cycles on fewer and fewer vertices until only $d$-complete cycles remain.  It is shown in \cite{Con13} that these are the $d$-dimensional cycles on the smallest number of vertices.  The notion of a $d$-chorded simplicial complex generalizes the graph theoretic notion of a chordal graph.  In particular a $1$-chorded complex is a chordal graph and conversely.  See Figure \ref{fig:2_chorded} for an example of a $2$-chorded simplicial complex.  This complex is comprised of a $2$-dimensional cycle, the hollow octahedron, with a chord set shown in a darker shading.

\begin{figure}[h!]
{\centering
    \includegraphics[height=1.15in]{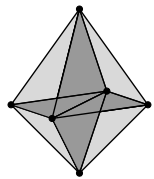}
    \caption {A $2$-chorded simplicial complex} \label{fig:2_chorded}

}
\end{figure}

The following lemma from \cite{ConFar12} demonstrates that a $d$-dimensional cycle which is the support complex of a $d$-boundary has a chord set.

\begin{lem} [Connon and Faridi \cite{ConFar12}] \label{lem:bdry_cycles_have_chord_sets}
Let $\Omega$ be a face-minimal $d$-dimensional cycle with vertex set $V$ that is not $d$-complete in a simplicial complex $\Gamma$.  If, over $\Z_2$, $\Omega$ is the support complex of a $d$-boundary of faces of $\Gamma_{V}$ then $\Omega$ has a chord set in $\Gamma$.
\end{lem}

Recall that the {\bf clique complex} of a graph $G$, denoted $\Delta(G)$, is the simplicial complex on the same vertex set as $G$ whose facets are given by the vertices in the maximal complete subgraphs of $G$.  In fact, Fr\"oberg originally gave Theorem \ref{thm:Frob_thm} in terms of the Stanely-Reisner ideal of the clique complex of a graph.  It is not hard to see that this ideal is equivalent to the edge ideal of a graph's complement.

\begin{thm}[Fr\"oberg \cite{Fr90}] \label{thm:Frob_original}
If a graph $G$ is chordal then the Stanley-Reisner ideal of $\Delta(G)$ has a $2$-linear resolution over any field.  Conversely, if the Stanley-Reisner ideal of a simplicial complex $\Gamma$ has a $2$-linear resolution over any field, then $\Gamma = \Delta(\Gamma^{[1]})$ and $\Gamma^{[1]}$ is chordal.
\end{thm}

There exists a similar notion to the clique complex in higher dimensions.

\begin{definition} [{\bf $d$-closure}]
The {\bf $d$-closure} of a pure $d$-dimensional simplicial complex $\Gamma$, denoted $\Delta_d(\Gamma)$, is the simplicial complex on $V(\Gamma)$ whose faces are given in the following way:
\begin{itemize}
\item the $d$-faces of $\Delta_d(\Gamma)$ are exactly the $d$-faces of $\Gamma$
\item all subsets of $V(\Gamma)$ with at most $d$ elements are faces of $\Delta_d(\Gamma)$
\item a subset of $V(\Gamma)$ with more than $d+1$ elements is a face of $\Delta_d(\Gamma)$ if and only if all of its subsets of $d+1$ elements are faces of $\Gamma$.
\end{itemize}
\end{definition}

The $d$-closure of $\Gamma$ is also called the {\bf complex of $\Gamma$} \cite{Em10} and the {\bf clique complex of $\Gamma$} \cite{MYZ12}.  We use the term $d$-closure to keep track of the dimension at which the operation is applied.  See Figure \ref{fig:2closure} for an example of $2$-closure.

\begin{figure}[h]
\centering
\subfloat[$\Gamma = \langle abc, abd, acd, bcd, bce, cde \rangle$]{\makebox[6cm]{
            \includegraphics[height=1.3in]{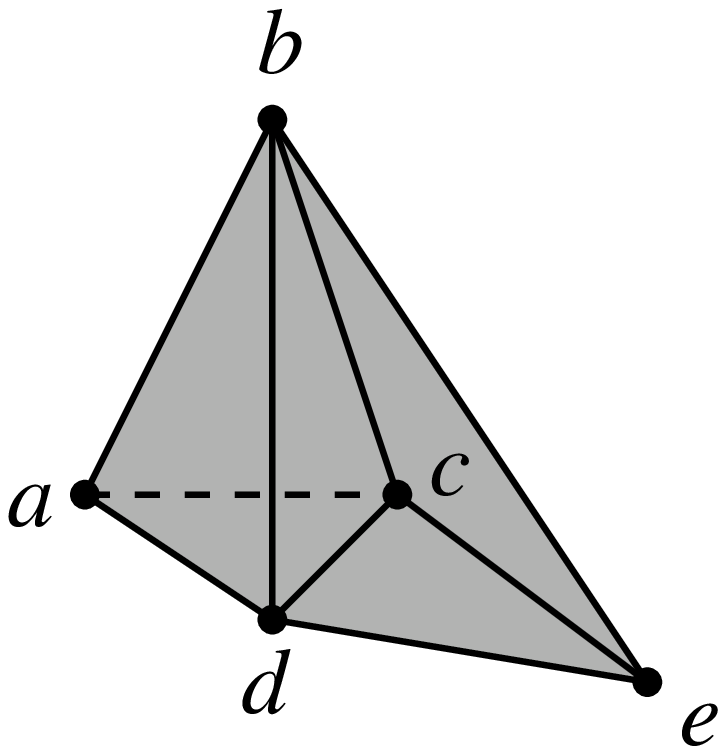}}
            \label{fig:2closure1}} \qquad \qquad
\subfloat[$\Delta_2(\Gamma) = \langle abcd, bce, cde, ae \rangle$]{\makebox[6cm]{
	\includegraphics[height=1.3in]{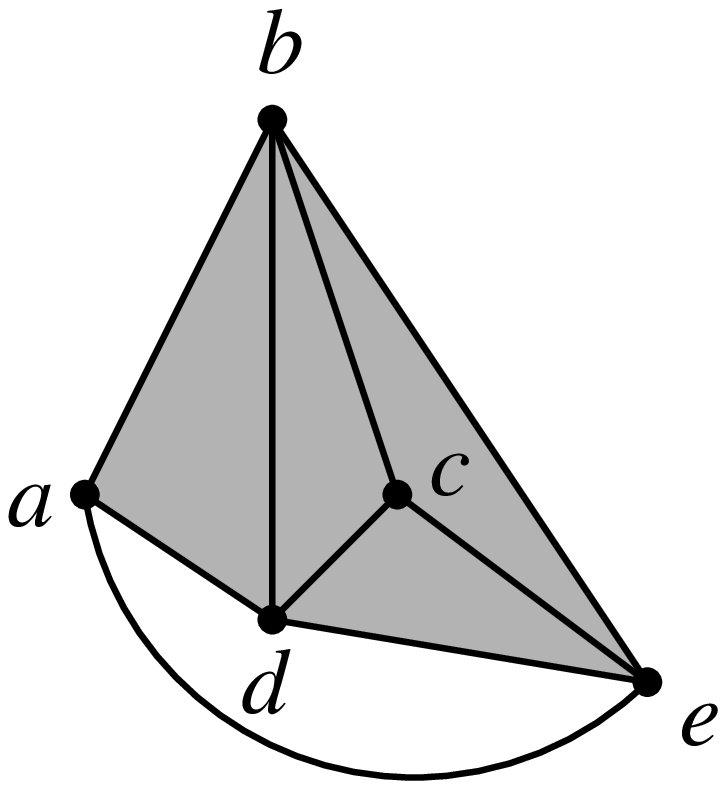}}
	\label{fig:2closure2}}
\caption{$2$-closure} \label{fig:2closure}
\end{figure}

The following lemma explains the results of subsequent applications of the closure operation on different dimensions.

\begin{lem} \label{lem:m_n_closure}
Let $\Gamma$ be a pure $n$-dimensional simplicial complex.
\begin{enumerate}
\item If $m<n$ then $\Delta_m(\Delta_n(\Gamma)^{[m]})$ is a simplex.
\item If $m=n$ then $\Delta_m(\Delta_n(\Gamma)^{[m]})=\Delta_n(\Gamma)$.
\item If $m>n$ then $\Delta_m(\Delta_n(\Gamma)^{[m]})^{[t]}= \Delta_n(\Gamma)^{[t]}$ for all $t \geq m$.
\end{enumerate}
\end{lem}

\begin{proof} \leavevmode
\begin{enumerate}
\item If $m<n$ then $\Delta_n(\Gamma)^{[m]}$ is $m$-complete as the $n$-closure adds all faces of dimension less than $n$.  Therefore by the definition of $m$-closure the set of all vertices of $\Delta_n(\Gamma)^{[m]}$ is a face of $\Delta_m(\Delta_n(\Gamma)^{[m]})$ and so $\Delta_m(\Delta_n(\Gamma)^{[m]})$ is a simplex.
\item If $m=n$ then by the nature of $n$-closure
           \[
                \Delta_m(\Delta_n(\Gamma)^{[m]}) = \Delta_n(\Delta_n(\Gamma)^{[n]}) = \Delta_n(\Gamma).
           \]
\item Let $m>n$ and let $F \in \Delta_n(\Gamma)^{[t]}$.  Then every subset $A$ of $F$ of size $m+1 \leq t+1$ is also a face of $\Delta_n(\Gamma)^{[t]}$ so $A \in \Delta_n(\Gamma)^{[m]}$.  Therefore $F \in \Delta_m(\Delta_n(\Gamma)^{[m]})^{[t]}$.

    Conversely, if $F \in \Delta_m(\Delta_n(\Gamma)^{[m]})^{[t]}$ then all subsets of $F$ of size $m+1 \leq t+1$ belong to $\Delta_n(\Gamma)^{[m]}$.  Therefore $\Delta_n(\Gamma)^{[m]}_F$ is $m$-complete.  Thus all subsets of $F$ of size $n+1 < m+1$ are in $\Delta_n(\Gamma)^{[m]}$ which means they are $n$-faces of $\Gamma$.  Hence by the definition of $n$-closure $F \in \Delta_n(\Gamma)^{[m]}$.
\end{enumerate}
\end{proof}

In the next proposition we see that when the minimal non-faces of a simplicial complexes are all of the same dimension $d$ then the complex is the $d$-closure of its pure $d$-skeleton.

\begin{prop} [Connon and Faridi \cite{ConFar12}] \label{prop:min_gen}
Let $\Gamma$ be a simplicial complex.  The ideal $\N(\Gamma)$ is minimally generated in degree $d+1$ if and only if $\Gamma = \Delta_d(\Gamma^{[d]})$.
\end{prop}

In \cite{ConFar12} we were able to show the following theorem which gives a necessary combinatorial condition for a Stanley-Reisner ideal to have a linear resolution over fields of characteristic $2$.  It is a generalization of one direction of Theorem \ref{thm:Frob_original} in the case of fields having characteristic $2$.

\begin{thm}[Connon and Faridi \cite{ConFar12}]\label{thm:one_dir}
Let $\Gamma$ be a simplicial complex, let $k$ be any field of characteristic $2$ and let $d \geq 1$.  If $\N(\Gamma)$ has a $(d+1)$-linear resolution over $k$ then $\Gamma = \Delta_d(\Gamma^{[d]})$ and $\Gamma^{[d]}$ is $d$-chorded.
\end{thm}

The converse of Theorem \ref{thm:one_dir} does not hold.  The following is a counterexample.

\begin{example} \label{ex:chorded_no_LR}
Let $\Gamma$ be the pure $2$-dimensional simplicial complex on the vertex set $\{x_0,\ldots,x_5\}$ whose minimal non-faces are $\{x_0,x_1,x_2\}$ and $\{x_3,x_4,x_5\}$.  The complex $\Gamma$ is a $2$-chorded simplicial complex and the Stanley-Reisner ideal of the $3$-dimensional simplicial complex $\Delta_2(\Gamma)$ does not have a linear resolution over $\Z_2$.  The pure $3$-skeleton of $\Delta_2(\Gamma)$ is a $3$-dimensional cycle with no chord set which is not $3$-complete and we have $\tilde{H}_3(\Delta_2(\Gamma);\Z_2)\neq 0$.
\end{example}

In the next section we determine which $d$-chorded complexes have $d$-closures which do not have $(d+1)$-linear resolutions in characteristic $2$.  By doing this we give a necessary and sufficient combinatorial condition for an ideal to have a linear resolution over a field of characteristic $2$.


\section{A combinatorial criterion for linear resolution in characteristic $2$} \label{sec:criterion_char2}

As we can see from Theorem \ref{thm:Frob_hom_char} for a square-free monomial ideal to have a linear resolution its Stanley-Reisner complex must have vanishing simplicial homology in all but one dimension.  Theorem \ref{thm:one_dir} shows that in characteristic $2$ this corresponds to a pure complex that is $d$-chorded, where $d$ is the dimension of the complex.

Conversely, in order to show that a particular class of simplicial complexes have Stanley-Reisner ideals with linear resolutions we must show that the simplicial homology of these complexes vanishes in the right dimensions.  Consider any pure $d$-dimensional simplicial complex $\Gamma$.  We know that $\Delta_d(\Gamma)$ contains all possible faces of dimension less than $d$.  This means that $\Delta_d(\Gamma)$ cannot have any non-zero simplicial homology in dimensions less than $d-1$.  If we assume that $\Gamma$ is $d$-chorded then we may further show the following lemma which essentially says that a $d$-dimensional cycle is the support complex of $d$-boundary.

\begin{lem}[Connon and Faridi \cite{ConFar12}] \label{lem:dchorded_boundary}
The sum of the $d$-faces of a $d$-dimensional cycle $\Omega$ in a $d$-chorded simplicial complex $\Gamma$ forms a $d$-boundary on $V(\Omega)$ in $\Delta_d(\Gamma)$ over $\Z_2$.
\end{lem}

The implication of Lemma \ref{lem:dchorded_boundary} is that when $\Gamma$ is $d$-chorded $\Delta_d(\Gamma)$ has vanishing homology in dimension $d$ and so altogether we have the following statement.

\begin{prop} [Connon and Faridi \cite{ConFar12}] \label{prop:chorded_dhom}
For any $d$-chorded simplicial complex $\Gamma$ and any field $k$ of characteristic $2$ we have $\tilde{H}_i(\Delta_d(\Gamma)_W;k)=0$ for all $W \subseteq V(\Gamma)$, $0 \leq i \leq d-2$ and $i=d$.
\end{prop}

As we can see from Example \ref{ex:chorded_no_LR} it is not necessarily the case that the upper-level homology groups of the $d$-closure of a $d$-chorded complex vanish.  In examples such as this the Stanley-Reisner ideal of the $d$-closure will not have a linear resolution.  In these cases the $d$-closure of the complex has a pure $m$-skeleton which is not $m$-chorded for some $m>d$.   When we require these $m$-skeletons to be $m$-chorded we obtain a necessary and sufficient condition for linear resolution over fields of characteristic $2$.

\begin{thm}[{\bf Criterion for a linear resolution I}] \label{thm:criterion}
Let $I$ be generated by square-free monomials of degree $d+1$.  The following are equivalent:
\begin{enumerate}[a)]
\item $I$ has a linear resolution over fields of characteristic $2$.
\item $\N(I)$ is chorded.
\item $\N(I)^{[m]}$ is $m$-chorded for all $m \geq d$.
\item $\Delta_d(\overline{\F(I)}_d)$ is chorded.
\item $\Delta_d(\overline{\F(I)}_d)^{[m]}$ is $m$-chorded for all $m\geq d$.
\end{enumerate}
\end{thm}

\begin{proof}
Let $\Gamma = \N(I)$ and let $\Upsilon = \overline{\F(I)}_d$.

 \bigskip

 \noindent a) $\Rightarrow$ b)  Suppose that $I$ has a linear resolution over any field of characteristic $2$.  By Theorem \ref{thm:one_dir} we know that  $\Gamma = \Delta_d(\Gamma^{[d]})$ and $\Gamma^{[d]}$ is $d$-chorded.  We also know that $\Gamma$ is $m$-complete for all $m<d$ by the definition of $d$-closure.  Therefore it follows from \cite[Remark 4.3]{ConFar12} that $\Gamma^{[m]}$ is $m$-chorded for $m< d$.

Let $m>d$ and let $\Omega$ be any face-minimal, non-$m$-complete $m$-dimensional cycle in $\Gamma^{[m]}$.  By Proposition \ref{prop:ddimcycle_is_dcycle} we know that $\Omega$ is the support complex of a homological $m$-cycle over $\Z_2$.  The ideal $I$ has a linear resolution over $\Z_2$ and so we know that $\tilde{H}_m(\Gamma_{V(\Omega)};\Z_2) =0$ by Theorem \ref{thm:Frob_hom_char}.  Thus $\Omega$ is also the support complex of an $m$-boundary of faces of $\Gamma_{V(\Omega)}$.  Therefore by Lemma \ref{lem:bdry_cycles_have_chord_sets} we know that $\Omega$ has a chord set in $\Gamma_{V(\Omega)}$.  Hence $\Gamma^{[m]}$ is $m$-chorded.  Therefore $\Gamma^{[m]}$ is $m$-chorded for all $1 \leq m \leq \dim \Gamma$ and so $\Gamma$ is chorded.

\bigskip

\noindent b) $\Rightarrow$ c) This is clear.

\bigskip

\noindent c) $\Rightarrow$ a) Suppose that $\Gamma^{[m]}$ is $m$-chorded for all $m \geq d$.  Since $I$ is generated by square-free monomials of degree $d+1$ then $\Gamma=\Delta_d(\Gamma^{[d]})$ by Proposition \ref{prop:min_gen}.  Therefore by Proposition \ref{prop:chorded_dhom} we know that for all $W \subseteq V(\Gamma)$ we have $\tilde{H}_i(\Gamma_W; k)=0$ for $0 \leq i \leq d-2$ and $i=d$.

Let $m > d$ and let $W \subseteq V(\Gamma)$.  We would like to show that $\tilde{H}_m(\Gamma_W; k)=0$.   By assumption $\Gamma^{[m]}$ is $m$-chorded.  Therefore by Proposition \ref{prop:chorded_dhom} we know that $\tilde{H}_m(\Delta_m(\Gamma^{[m]})_W;k)=0$.  Furthermore, by Lemma \ref{lem:m_n_closure} we have
\[
    \Delta_m(\Gamma^{[m]})^{[t]} = \Delta_m(\Delta_d(\Gamma^{[d]})^{[m]})=\Delta_m(\Gamma^{[m]})^{[t]}
\]
for all $t \geq m$.  Thus the $m$-faces and the $m+1$-faces of $\Delta_m(\Gamma^{[m]})_W$ and $\Gamma_W = \Delta_d(\Gamma^{[d]})_W$ are equivalent. Therefore we have
\[
    \tilde{H}_m(\Gamma_W; k)=\tilde{H}_m(\Delta_m(\Gamma^{[m]})_W;k)=0
\]
for all $m>d$.  Consequently $\tilde{H}_m(\Gamma_W; k)=0$ for all $m > d$.  Hence $I$ has a $(d+1)$-linear resolution by Theorem \ref{thm:Frob_hom_char}.

\bigskip

\noindent b) $\Leftrightarrow$ d)  It is easy to see that the $d$-complement of $\F(I)$ is equal to the pure $d$-skeleton of $\N(I)=\Gamma$.  Thus $\Upsilon = \Gamma^{[d]}$ and so $\Delta_d(\Upsilon)$ is chorded if and only $\Delta_d(\Gamma^{[d]})=\N(I)$ is chorded.

\bigskip

\noindent c) $\Leftrightarrow$ e) As before, $\Upsilon = \Gamma^{[d]}$ and so $\Delta_d(\Upsilon)^{[m]}$ is $m$-chorded for all $m \geq d$ if and only $\Delta_d(\Gamma^{[d]})^{[m]}$ is $m$-chorded for all $m \geq d$.
\end{proof}

The condition for $(d+1)$-linear resolution in Theorem \ref{thm:criterion} requires checking that every non-$m$-complete, face minimal $m$-dimensional cycle in $\N(I)^{[m]}$ has a chord set for all $m \geq d$ which can be tedious.  However our next result shows that in most cases assuming that $\N(I)^{[d]}$ is $d$-chorded suffices.  The only possible obstruction to this implication is the presence of an $m$-dimensional cycle of a very special form.  In general we expect these types of cycles to occur infrequently.  Thus to check for a linear resolution we need only verify that $\N(I)^{[d]}$ is $d$-chorded and that any cycles of this special nature have chord sets.

\begin{thm}\label{thm:skeletons_n_chorded}
Let $\Gamma$ be a $d$-chorded simplicial complex.  Suppose that for all $m >d$ each $1$-complete, face-minimal, non-$m$-complete $m$-dimensional cycle in $\Delta_d(\Gamma)$ has a chord set in $\Delta_d(\Gamma)$. Then $\Delta_{d}(\Gamma)$ is chorded.
\end{thm}

\begin{proof}
By the nature of the $d$-closure we know that $\Delta_d(\Gamma)$ is $t$-complete for all $t<d$.  Thus $\Delta_{d}(\Gamma)^{[t]}$ is $t$-chorded for all $t<d$ \cite[Remark 4.3]{ConFar12}.

For the remaining cases we will use induction on $t$.  When $t=d$ we have $\Delta_d(\Gamma)^{[d]}=\Gamma$.  Since $\Gamma$ is $d$-chorded by assumption this proves the base case.

Now suppose that $t>d$ and we know that $\Delta_d(\Gamma)^{[n]}$ is $n$-chorded for all $n<t$.  Let $\Omega$ be a face-minimal $t$-dimensional cycle that is not $t$-complete in $\Delta_d(\Gamma)^{[t]}$.  We would like to show that $\Omega$ has a chord set in $\Delta_d(\Gamma)^{[t]}$.  If $\Omega$ is $1$-complete then by assumption $\Omega$ has a chord set in $\Delta_d(\Gamma)^{[t]}$, and so we may assume that $\Omega$ is not $1$-complete.  Then there exist $u,v \in V(\Omega)$ such that $u$ and $v$ are not contained in the same $t$-face of $\Omega$.

Let $F_1,\ldots,F_k$ be the $t$-faces of $\Omega$ containing $v$. By Proposition \ref{prop:dcycle_contains_smallercycle} we know that the $(t-1)$-path-connected components of $\langle F_1\setminus \{v\},\ldots,F_k\setminus \{v\}\rangle$ are $(t-1)$-dimensional cycles.  Call these cycles $\Phi_1,\ldots,\Phi_m$.  For each $i \in \{1,\ldots,m\}$ let $P_i \subseteq \{1,\ldots,k\}$ be such that $F_j \setminus \{v\} \in \Phi_i$ if and only if $j \in P_i$.  Since for each $j$ the face $F_j \setminus \{v\}$ must belong to exactly one of $\Phi_1,\ldots,\Phi_m$, the sets $P_1,\ldots,P_m$ form a partition of $\{1,\ldots,k\}$.

The complex $\Delta_d(\Gamma)^{[t-1]}$ is $(t-1)$-chorded by assumption and so by Lemma \ref{lem:dchorded_boundary} the sum of the $(t-1)$-faces of $\Phi_i$ form a $(t-1)$-boundary in $\Delta_{t-1}(\Delta_d(\Gamma)^{[t-1]})$ on $V(\Phi_i)$ over $\Z_2$ for each $i$ and therefore in $\Delta_d(\Gamma)$ by Lemma \ref{lem:m_n_closure}.

Hence for each $i$ there exist $t$-faces $A^i_{1},...,A^i_{\ell_i}$ in $\Delta_d(\Gamma)_{V(\Phi_i)}$ such that
\begin{equation}\label{eq:boundary}
    \partial_t \left( \sum_{j=1}^{\ell_i} A^i_{j} \right) = \sum ((t-1)\textrm{-faces of }\Phi_i).
\end{equation}
Without loss of generality we may assume that the choice of $A^i_{1},...,A^i_{\ell_i}$ is minimal in the sense that for no strict subset of $A^i_{1},...,A^i_{\ell_i}$ is (\ref{eq:boundary}) satisfied.  Let $\Omega_i$ be the simplicial complex whose facets are $\{F_j | j \in P_i\} \cup \{A^i_{1},...,A^i_{\ell_i} \}$.

By Proposition \ref{prop:cone_cycle} we know that for $1 \leq i \leq m$ each $\Omega_i$ is a $t$-dimensional cycle and $V(\Omega_i)\subsetneq V(\Omega)$ as $u \notin V(\Omega_i)$.  Since $\Omega$ is a face-minimal $t$-dimensional cycle, each $\Omega_i$ must contain at least one $t$-face which is not in $\Omega$.  We collect all of these $t$-faces in the non-empty set $C$:
\[
    C = \{A_j^i \notin \Omega \ | \ 1 \leq i \leq m, \ 1 \leq j \leq \ell_i\}.
\]
We would like to show that $C$ is a chord set of $\Omega$ in $\Delta_d(\Gamma)^{[t]}$.

Consider the collection of $t$-faces in $\Omega$ and those in $\Omega_1,\ldots,\Omega_m$ with repeats.  Let $H_1,\ldots,H_s$ be the $t$-faces in this collection which appear an odd number of times so that over $\Z_2$ we have
\begin{equation} \label{eq:sum_faces}
    \sum_{i=1}^s H_i = \sum (t\textrm{-faces of } \Omega) + \sum_{i=1}^m\sum (t\textrm{-faces of } \Omega_i).
\end{equation}

Since $\Omega$ and $\Omega_1,\ldots,\Omega_m$ are all $t$-dimensional cycles, by Proposition \ref{prop:ddimcycle_is_dcycle} they correspond to homological $t$-cycles over $\Z_2$.  Therefore by (\ref{eq:sum_faces}) over $\Z_2$ we have,
\[
  \partial_t\left(\sum_{i=1}^s H_i\right)
   = \partial_t \left( \sum (t\textrm{-faces of } \Omega)\right) + \sum_{i=1}^m\partial_t\left(\sum (t\textrm{-faces of } \Omega_i) \right) = 0.
\]

Hence the $t$-path-connected components of the simplicial complex $\langle H_1,\ldots,H_s\rangle$ are $t$-dimensional cycles by Proposition \ref{prop:ddimcycle_is_dcycle}.  Call these cycles $\Omega_{m+1},\ldots,\Omega_M$.  We would like to show that our set $C$ is a chord set that breaks $\Omega$ into the cycles $\Omega_1,\ldots,\Omega_M$.  By (\ref{eq:sum_faces}), after rearranging the sums, over $\Z_2$ we have
\[
    \sum (t\textrm{-faces of }\Omega) = \sum_{i=1}^M \sum(t\textrm{-faces of }\Omega_i).
\]
By noticing that the set $C$ is exactly those $t$-faces on the right-hand side of this equation which do not belong to $\Omega$ we can see that properties 2 and 3 of a chord set hold for $C$.  Also, it is clear from our construction that all $t$-faces of both $\Omega$ and of $C$ appear in at least one of the $\Omega_i$'s.  Therefore property 1 of a chord set holds for the set $C$.

Now since none of $\Omega_1,\ldots,\Omega_m$ contain $u$ by construction we have $|V(\Omega_i)|<|V(\Omega)|$ for all $1 \leq i \leq m$.  We would like to show that none of $\Omega_{m+1},\ldots,\Omega_M$ contain $v$.  Recall that $\Phi_1,\ldots,\Phi_m$ are the $(t-1)$-path-connected components of $\langle F_1\setminus \{v\},\ldots,F_k\setminus \{v\} \rangle$ and so no two such distinct components could share a face of the form $F_i\setminus\{v\}$.  Thus each face $F_i$ appears in only one of the cycles $\Omega_1,\ldots,\Omega_m$.  Each such $F_i$ is also a face of $\Omega$ and so by our choice of $H_1,\ldots,H_s$ we know that we cannot have $F_i = H_j$ for any $i \in \{1,\ldots,k\}$ and $j \in \{1,\ldots,s\}$.  Therefore, by the construction of the cycles $\Omega_{m+1},\ldots,\Omega_M$ we know that none of the $F_i$'s appear in any of these cycles.  Recall that $F_1,\ldots,F_k$ are the only $t$-faces of $\Omega$ that contain $v$ and none of the $t$-faces of $C$ contain $v$ since they are subsets of $\bigcup_{i=1}^m V(\Phi_i)$.  It follows that none of $\Omega_{m+1},\ldots,\Omega_M$ contain $v$.  This implies that $|V(\Omega_i)|<|V(\Omega)|$ for all $m+1 \leq i \leq M$.  Thus property 4 of a chord set is also satisfied by $C$ and hence $\Delta_d(\Gamma)^{[t]}$ is $t$-chorded.  Hence $\Delta_d(\Gamma)$ is chorded.
\end{proof}

As a consequence of Theorems \ref{thm:criterion} and \ref{thm:skeletons_n_chorded} we have the following theorem.

\begin{thm}[{\bf Criterion for a linear resolution II}] \label{thm:criterion2}
Let $I$ be generated by square-free monomials of degree $d+1$.  Then $I$ has a linear resolution over any field of characteristic $2$ if and only if $\N(I)^{[d]}$ is $d$-chorded and for $m > d$ each $1$-complete, face-minimal, non-$m$-complete $m$-dimensional cycle in $\N(I)$ has a chord set in $\N(I)$.
\end{thm}

From Theorems \ref{thm:one_dir} and \ref{thm:criterion2} we conclude that for any square-free monomial ideal $I$ generated in degree $d+1$, if $I$ has no linear resolution then either $\N(I)^{[d]}$ is not $d$-chorded or for some $m>d$ there exists a $1$-complete face-minimal non-$m$-complete $m$-dimensional cycle in $\N(I)$ which has no chord set.  Example \ref{ex:chorded_no_LR} gives an instance of a complex $\N(I)$ in the latter case.

In the next section we prove that in the $1$-dimensional case, such obstructions to linear resolution do not exist.  In particular if $\Gamma^{[1]}$ is $1$-chorded then in $\Delta_1(\Gamma^{[1]})$ all $1$-complete $m$-dimensional cycles lie in $m$-complete induced subcomplexes which are $m$-chorded and consequently such cycles have chord sets.  This leads us to a new, combinatorial proof of Theorem \ref{thm:Frob_thm} in characteristic $2$.


\section{A new proof of Fr\"oberg's Theorem in characteristic $2$} \label{sec:Frob_proof}
In the proof of Theorem \ref{thm:Frob_original} in \cite{Fr90} Fr\"oberg shows that the simplicial homology of the clique complex of a chordal graph vanishes on all levels greater than zero.  He does so by dismantling the graph at a complete subgraph and then applying the Mayer-Vietoris sequence on the resulting dismantled clique complex.  This is a very clean and elegant method for demonstrating that all upper-level homologies are zero.  However, this technique gives no intuitive sense as to why it should be the case that filling in complete subgraphs of a chordal graph produces a simplicial complex with no homology on higher levels.  A chordal graph may contain complete subgraphs on any number of vertices and so the clique complex may have faces of any dimension.  The question is why the addition of these higher-dimensional faces doesn't introduce any new homology.  The following theorem, together with Proposition \ref{prop:chorded_dhom}, answers this question, from a combinatorial point of view, in the case that the field of interest has characteristic $2$.

\begin{lem} \label{thm:delta_chorded}
A graph $G$ is chordal if and only if $\Delta_1(G)$ is chorded.
\end{lem}

\begin{proof}
Since $G$ is chordal it is $1$-chorded.  Let $m>1$ and let $\Omega$ be a $1$-complete, face-minimal $m$-dimensional cycle in $\Delta_1(G)$ that is not $m$-complete.  Then $\Delta_1(G)_{V(\Omega)}$ is a $(|V(\Omega)|-1)$-simplex and hence $\Delta_1(G)_{V(\Omega)}^{[m]}$ is $m$-chorded \cite[Remark 4.3]{ConFar12}.  Therefore $\Omega$ has a chord set in $\Delta_1(G)$ and by Theorem \ref{thm:skeletons_n_chorded}, $\Delta_1(G)$ is chorded.

Conversely, $\Delta_1(G)^{[1]}$ is $1$-chorded and $\Delta_1(G)^{[1]}=G$ so $G$ is a chordal graph.
\end{proof}

This gives us a new proof of Fr\"oberg's theorem over fields of characteristic $2$ using the notion of $d$-chorded complexes.

\begin{thm}
If $G$ is chordal then $\N(\Delta_1(G))$ has a $2$-linear resolution over any field of characteristic $2$.  Conversely, if $\N(\Gamma)$ has a $2$-linear resolution over a field of characteristic $2$ then $\Gamma = \Delta_1(\Gamma^{[1]})$ and $\Gamma^{[1]}$ is chordal.
\end{thm}

\begin{proof}
Let $G$ be a chordal graph and let $k$ be a field of characteristic $2$.  To show that $\N(\Delta_1(G))$ has a 2-linear resolution over $k$ we need to show that $\tilde{H}_i(\Delta_1(G)_W; k)=0$ for all $i \geq 1$ and all $W \subseteq V(G)$.  Let $d \geq 1$ and let $W \subseteq V(G)$.  We know by Theorem \ref{thm:delta_chorded} that $\Delta_1(G)^{[d]}$ is $d$-chorded.  Therefore by Proposition \ref{prop:chorded_dhom} we know that $\tilde{H}_d(\Delta_d(\Delta_1(G)^{[d]})_W;k)=0$.  However it is easy to see that $\tilde{H}_d(\Delta_d(\Delta_1(G)^{[d]})_W;k)=\tilde{H}_d(\Delta_1(G)_W;k)$ since the $d$-faces and the $(d+1)$-faces of the two complexes are equivalent by Lemma \ref{lem:m_n_closure}.  Hence we have $\tilde{H}_d(\Delta_1(G)_W; k)=0$ for all $d \geq 1$.  Therefore $\N(\Delta_1(G))$ has a 2-linear resolution by Theorem \ref{thm:Frob_hom_char}.

The converse follows by Theorem \ref{thm:one_dir} and by the equivalence of the notions of chordal and 1-chorded.
\end{proof}

\bibliography{chorded_paper_bib}

\begin{thebibliography}{10}

\bibitem{Con13}
E.~Connon.
\newblock On $d$-dimensional cycles and the vanishing of simplicial homology.
\newblock Preprint, 2013.

\bibitem{ConFar12}
E.~Connon and S.~Faridi.
\newblock Chorded complexes and a necessary condition for a monomial ideal to
  have a linear resolution.
\newblock {\em Journal of Combinatorial Theory, Series A}, to appear.

\bibitem{EagRein98}
J.~Eagon and V.~Reiner.
\newblock Resolutions of {S}tanley-{R}eisner rings and {A}lexander duality.
\newblock {\em J. Pure Appl. Algebra}, 130(3):265--275, 1998.

\bibitem{Em10}
E.~Emtander.
\newblock A class of hypergraphs that generalizes chordal graphs.
\newblock {\em Mathematica Scandinavica}, 106(1):50--66, 2010.

\bibitem{Fr85}
R.~Fr\"oberg.
\newblock Rings with monomial relations having linear resolutions.
\newblock {\em J. Pure Appl. Algebra}, 38:235--241, 1985.

\bibitem{Fr90}
R.~Fr\"oberg.
\newblock On {S}tanley-{R}eisner rings.
\newblock In {\em Topics in Algebra, \emph{Part II}}, volume~26, pages 57--70,
  Warsaw, 1990. Banach Center Publ., PWN.

\bibitem{HaVT08}
H.~T. H\`a and A.~Van~Tuyl.
\newblock Monomial ideals, edge ideals of hypergraphs, and their graded {B}etti
  numbers.
\newblock {\em Journal of Algebraic Combinatorics}, 27(2):215--245, 2008.

\bibitem{MNYZ12}
M.~Morales, A.~Nasrollah~Nejad, A.~A. Yazdan~Pour, and R.~Zaare-Nahandi.
\newblock Monomial ideals with $3$-linear resolutions.
\newblock 2012.
\newblock arXiv:1207:1789v1.

\bibitem{MYZ12}
M.~Morales, A.~A. Yazdan~Pour, and R.~Zaare-Nahandi.
\newblock Regularity and free resolution of ideals which are minimal to
  $d$-linearity.
\newblock 2012.
\newblock arXiv:1207:1790v1.

\bibitem{Wood11}
R.~Woodroofe.
\newblock Chordal and sequentially {C}ohen-{M}acaulay clutters.
\newblock {\em Electron. J. Combin.}, 18(1), 2011.
\newblock Paper 208.

\end{thebibliography}

\end{document}